\documentclass[preprint, number]{elsarticle}
\usepackage{amssymb}
\usepackage{amsmath}
\usepackage{graphicx} % Required for inserting images

\usepackage{vmargin,amsmath,amsfonts,amssymb,amsthm,xcolor,amscd,latexsym,xspace,enumerate}
\usepackage[colorlinks=true,linkcolor=blue,citecolor=red]{hyperref}
\usepackage{dirtytalk}
\usepackage{tikz}
\usepackage{pgfplots}
\pgfplotsset{compat = 1.15}

\definecolor{blue}{rgb}{0,0,0.8}
\definecolor{red}{rgb}{0.8,0,0}
\definecolor{darkgreen}{rgb}{0,0.6,0}
\definecolor{orange}{rgb}{0.98,0.5,0}

\newcommand{\ee}{\eta}

\newcommand{\ve}{\varepsilon}
\newcommand{\Star}[1]{{#1}^{*}}

\newcommand{\supp}{\operatorname{supp}}

\newcommand{\ext}{\operatorname{ext}}
\newcommand{\Int}{\operatorname{int}}

\newtheorem{theorem}{Theorem}
\newtheorem{corollary}[theorem]{Corollary}
\newtheorem{lemma}[theorem]{Lemma}
\newtheorem{proposition}[theorem]{Proposition}

\newtheorem{definition}[theorem]{Definition}
\newtheorem{remark}[theorem]{Remark}
\newtheorem{example}[theorem]{Example}

\begin{document}

\begin{frontmatter}

\title{On extremal problems of Delsarte type \\ for positive definite functions on LCA groups \\[2mm]
{\small To the memory of Bent Fuglede whose ideas have a lasting influence on us.} }

\author[1]{E. E. Berdysheva}
\ead{elena.berdysheva@uct.ac.za}
\author[1]{M.~D.~Ramabulana}
\ead{rmbmit001@myuct.ac.za}
\author[3]{\corref{cor1}Sz. Gy. R\'ev\'esz}
\ead{revesz.szilard@renyi.hu}
\cortext[cor1]{Corresponding author}

\affiliation[1]{organization={Department of Mathematics and Applied Mathematics, University of Cape Town},%Department and Organization
	addressline={Private Bag X1, 7701, Rondebosch},
	city={Cape Town},
	postcode={7700},
	country={South Africa}}

\affiliation[3]{organization={HUN-REN Alfr\'ed R\'enyi Institute of Mathematics},%Department and Organization
	addressline={Re\'altanoda utca 13-15},
	city={Budapest},
	postcode={1053},
	country={Hungary}}

\begin{abstract}
	A unifying framework for some extremal problems on locally compact Abelian groups is considered, special cases of which include the Delsarte and Tur\'an extremal problems.  A slight variation of the extremal problem is introduced and the different formulations are studied for equivalence. Extending previous work, a general result on existence of extremal functions for the new variant is proved under a certain general topological condition.
\end{abstract}

\begin{keyword}
	Locally compact Abelian groups, positive definite functions, extremal problems, Delsarte problem, Tur\'an problem
	
\MSC[2020] 43A35

\end{keyword}

\end{frontmatter}

\section{Introduction}

Let $G$ be a Hausdorff locally compact Abelian (LCA) group with identity $0$ and a fixed Haar measure $\lambda_{G}$. A function $f: G \to \mathbb{R}$ is positive definite if the inequality
\begin{equation*}
\sum_{i = 1}^{n}\sum_{j=1}^{n}c_{i}\overline{c_{j}}f(g_{i}-g_{j}) \ge 0
\end{equation*}
holds for all choices of $n \in \mathbb{N}$, $c_{i} \in \mathbb{C}$, and $g_{i} \in G$. Denote by $P_{1}(G)$ the collection of continuous positive definite real-valued functions $f: G \to \mathbb{R}$ such that $f(0)=1$. For a function $f: G \to \mathbb{R}$, denote by $f_{+}$, respectively, $f_{-}$ the positive part and the negative part of $f$, given by
\begin{equation*}
f_{+}(g) := \max{\{f(g), 0 \}} \hspace{2mm} \mbox{  and   } \hspace{2mm}  f_{-}(g) := \max{\{-f(g), 0\}} \hspace{2mm} \mbox{ for all } g \in G.
\end{equation*}

Denote by $\supp f$ its support, defined as $\supp f :=\overline{\{ g \in G: f(g) \ne 0\}}$.
In \cite{elena-szilard}, the following general setup for some extremal problems on a LCA group $G$ is introduced. Let $\Omega_{+}$ and $ \Omega_{-}$ be subsets of $G$\footnote{Note that in \cite{elena-szilard}, $\Omega_{+}$ and $\Omega_{-}$ are assumed to be open, a restriction we do not apply.}. Consider the collection
\begin{equation}\label{supportclass}
\mathcal{F}^{*}_{G}(\Omega_{+}, \Omega_{-}):= \Bigl\{f \in P_{1}(G) \cap L^1(G) : \supp f_{+} \subset \Omega_{+}, \supp f_{-} \subset \Omega_{-} \Bigr \},
\end{equation}
and the extremal problem of determining the following extremal value\footnote{We adopt the convention that the extremal value is $0$ if $\mathcal{F}^{*}_{G}(\Omega_{+}, \Omega_{-})$ is empty. The same will apply to \eqref{extremalproblem}.}:
\begin{equation}\label{supportedproblem}
\mathcal{C}^{*}_{G}(\Omega_{+}, \Omega_{-}):= \sup_{f \in \mathcal{F}^{*}_{G}(\Omega_{+},  \Omega_{-})} \int_{G}f\textup{d}\lambda_{G}.
\end{equation}

To motivate this general set-up, it is noted in \cite{elena-szilard} that when $\Omega_{-} =G$, writing $\Omega_{+} = \Omega$, the extremal problem \eqref{supportedproblem} specialises to the so-called Delsarte extremal problem, which concerns the extremal value:
\begin{align}\label{delsarte}
\mathcal{D}^{*}_{G}(\Omega) := \sup _{f \in \mathcal{F}^{*}_{G}(\Omega, G)}\int_{G}f\mbox{d}\lambda_{G}.
\end{align}
When $\Omega_{+} = \Omega_{-} = \Omega$, it specialises to the so-called Tur\'{a}n extremal problem\footnote{Actually, the problem was introduced well before Tur\'an by Siegel \cite{siegel} and, in a different setting, it goes back to Carath\'{e}odory \cite{caratheodory} and Fej\'{e}r \cite{fejer}. For the history of the problem and its somewhat inadequate name, see \cite{turan-szilard}.}, which concerns the extremal value:
\begin{align}\label{turan}
\mathcal{T}^{*}_{G}(\Omega) := \sup _{f \in \mathcal{F}_{G}(\Omega, \Omega)}\int_{G}f\mbox{d}\lambda_{G}.
\end{align}

The idea of studying this type of extremal problems in the general setting of LCA groups  goes back to \cite{kolrev}. However, we note that the extremal problem of maximising the $L^{2}$-norm was already considered in the group setting in \cite{domar}. Subsequently, extremal problems of this type on LCA groups have been studied in \cite{turan-szilard, elena-szilard, marcell-zsuzsa, ramabulana}. Whereas \cite{kolrev, turan-szilard, elena-szilard} derive estimates for the extremal values and relate them to packing in LCA groups, the theme in \cite{marcell-zsuzsa, ramabulana} is to study the existence of an extremal function. In \cite{marcell-zsuzsa}, the existence of an extremal function is proved for a collection of band-limited functions, i.e., functions whose Fourier transforms vanish outside some prescribed subset. In \cite{ramabulana}, existence is proved without the restriction that the functions in the collection are band-limited, generalising \cite[Proposition A.1]{cohnlaatsal} to arbitrary LCA groups.

Although the result of \cite{ramabulana} is quite general, the existence of an extremal function for the Tur\'{a}n and Delsarte extremal problems was known before \cite{ramabulana} for specific groups $G$ and specific subsets $\Omega_{+}$. For instance, for the Tur\'{a}n problem in $\mathbb{R}^d$ it was known that if $\Omega_{+}$ is a closed convex subset of $\mathbb{R}^{d}$ that tiles $\mathbb{R}^d$ by translations or when $\Omega_{+}$ is a closed ball in $\mathbb{R}^d$, then an extremal function exists (see \cite{elena-arestov-2, kolrev0, gorbachev1}). On the other hand, if $\Omega_{+}$ is, for instance, an open convex set in $\mathbb{R}^d$, an extremal function does not exist.

If one reformulates the extremal problem in \eqref{supportedproblem}, modifying the support condition and imposing a topological condition on the sets $\Omega_{+}$ and $\Omega_{-}$, and keeping all else the same, then we get a related extremal problem. To be more precise, and delve more into the subject of this paper, let $\Omega_{+}$ and $\Omega_{-}$ be subsets of $G$. Let us fix the following notation: for a real-valued function $f: G \to \mathbb{R}$ and $a,b \in \overline{\mathbb{R}}$, where $\overline{\mathbb{R}}$ denotes the extended real numbers, write $f^{-1}(a,b)$ for the preimage $f^{-1}\left((a,b)\right)$ of $(a,b)$ under $f$. Consider the collection
\begin{equation}\label{oursetting}
\mathcal{F}_{G}(\Omega_{+}, \Omega_{-}):= \Bigl\{f \in P_{1}(G) \cap L^1(G): f^{-1}(0,\infty) \subset \Omega_{+}, f^{-1}(-\infty,0) \subset \Omega_{-} \Bigr \},
\end{equation}
and the extremal value:
\begin{align}\label{extremalproblem}
\mathcal{C}_{G}(\Omega_{+}, \Omega_{-}):= \sup _{f \in \mathcal{F}_{G}(\Omega_{+}, \Omega_{-})}\int_{G}f\mbox{d}\lambda_{G}.
\end{align}

Note the resemblance between the function classes $\mathcal{F}^{*}_{G}(\Omega_{+}, \Omega_{-})$ and $\mathcal{F}_{G}(\Omega_{+}, \Omega_{-})$. Yet, the two function classes differ in that the conditions $\supp f_{+} \subset \Omega_{+}$ and $\supp f_{-} \subset \Omega_{-}$ are required for $\mathcal{F}^{*}_{G}(\Omega_{+}, \Omega_{-})$, whereas the weaker conditions that $f^{-1}(0, \infty) \subset \Omega_{+}$ and  $f^{-1}(-\infty, 0)  \subset \Omega_{-}$ are required for $\mathcal{F}_{G}(\Omega_{+}, \Omega_{-})$. One of the aims of this paper is to study the extent to which these two formulations of the extremal problems are equivalent. However, our main goal is to show general existence results for these slightly reformulated versions of our general extremal problems, see Theorem \ref{generalcase}. Actually, we show existence under a mild topological condition on the boundary $\partial \Omega_{\pm}$ of $\Omega_{\pm}$. To state the condition, recall that the exterior $\ext \Omega$ of a subset $\Omega$ in a topological space $X$ is the set $X \setminus \overline{\Omega}$ where $\overline{\Omega}$ stands for the closure of the set $\Omega$.

\begin{definition}

Let $\Omega$ be a subset of a topological space $X$. We call $\Omega$ a boundary-coherent set if $\partial \Omega \subset \overline{\ext \Omega}$.

\end{definition}

The definition above is equivalent to saying that  $\partial \Omega \subset \partial{(\ext \Omega)}$. This property means that boundary points of $\Omega$ can be approximated arbitrarily closely from the exterior of $\Omega$. Equivalently, we could have said that the complement of $\Int \Omega$ is fat where $\Int{\Omega}$ stands for the interior of the set $\Omega$.

We show the existence of an extremal function for the extremal problem \eqref{extremalproblem} under the condition that $\Omega_{\pm}$ are boundary-coherent.

Having introduced the extremal problem \eqref{extremalproblem}, we have the corresponding Delsarte and Tur\'{a}n problems:
\begin{align}\label{delsarte'}
\mathcal{D}_{G}(\Omega) := \sup _{f \in \mathcal{F}_{G}(\Omega, G)}\int_{G}f\mbox{d}\lambda_{G}, \hspace{0.5cm} \mathcal{T}_{G}(\Omega) := \sup _{f \in \mathcal{F}_{G}(\Omega, \Omega)}\int_{G}f\mbox{d}\lambda_{G}.
\end{align}

\section{Notation and Preliminaries}

Let $C(G)$ be the space of continuous functions on $G$ to the complex numbers and put on it the topology of uniform convergence on compact sets. The dual group $\widehat{G}$ of $G$ consists of continuous homomorphisms of $G$ to the multiplicative group $\mathbb{T} = \{z \in \mathbb{C}: |z| = 1 \}$. With respect to pointwise operations and the subspace topology that it inherits as a subset of $C(G)$, $\widehat{G}$ is a LCA group. For a measure space $X$ with a measure $\mu$, $L^p(X)$ denotes the usual Banach space of $p$-integrable complex-valued functions on $X$ for $1 \le p < \infty$, and the space of essentially bounded functions on $X$ for $p = \infty$. The Fourier transform of $f \in L^1(G)$ is defined by the formula:
\begin{equation*}
\widehat{f}(\chi) := \int_{G}f\overline{\chi}\mbox{d}\lambda_{G} \mbox{ for all } \chi \in \widehat{G}.
\end{equation*}

It is well-known that the Fourier transform of an integrable continuous positive definite function is non-negative. Moreover, since $0 \le \widehat{f}(0) = \int_{G}f\textup{d}\lambda_{G}$, it follows that the integral of an integrable continuous positive definite function is non-negative. Related to the notion of a positive definite function is that of an integrally positive definite function. Firstly, recall that the convolution $f \ast g$ of functions $f,g: G \to \mathbb{C}$ is defined as
\begin{equation*}
f \ast g (x) := \int_{G}f(y)g(x - y)\mbox{d}\lambda_{G}(y) \mbox{ for all } x \in G,
\end{equation*}
whenever the integral exists. Write $\Star{f}(g) := \overline{f(-g)}$. A function $f \in L^{\infty}(G)$ is integrally positive definite if the inequality
\begin{equation*}
\int_{G}(\Star{h} \ast h)f\mbox{d}\lambda_{G} = \int_{G}\int_{G}f(y - x)h(x)\overline{h(y)}\textup{d}\lambda_{G}(x)\textup{d}\lambda_{G}(y) \ge 0
\end{equation*}
holds for all continuous compactly supported complex valued functions $h$ on $G$. According to \cite[Theorem 1.7.3]{sasvari} an integrally positive definite function agrees locally almost everywhere with a continuous positive definite function. When $G$ is $\sigma$-compact, i.e., if $G$ can be written as a countable union of compact subsets, then locally almost everywhere is the same as almost everywhere. Therefore, on a $\sigma$-compact group an integrally positive definite function agrees almost everywhere with a continuous positive definite function.

For brevity, we write $K \Subset G$ to mean that $K$ is a compact subset of $G$.

\section{Extremal Problems and $\sigma$-compact Groups}

In this section, we show that the computation of the extremal values studied in this paper may be reduced to $\sigma$-compact LCA groups.  We record some preliminary lemmas.

\begin{lemma}\label{sigmacompactgeneration}
	Let $G$ be a LCA group, and let $(K_n)_{n \in \mathbb{N}}$ be a sequence of $\sigma$-compact subsets of $G$. The group $H$ generated by the union $\bigcup_{n \in \mathbb{N}}K_n$ is $\sigma$-compact. Furthermore, if at least one of the $K_{n}$'s has non-empty interior, then $H$ is open.
\end{lemma}

The above statement is standard, see \cite[Lemma 5]{ramabulana} and e.g. \cite[Proposition 1.2.1(c)]{deitmar}.

If $H$ is a subgroup of a LCA group $G$ and $\varphi: H \to \mathbb{R}$ is a function, its trivial extension is the function $\widetilde{\varphi}: G \to \mathbb{R}$, given by
\begin{equation*}
\widetilde{\varphi}(g) := \begin{cases}
\varphi(g) & \text{if } g \in H, \\
0  & \text{if } g \in G\symbol{92}H.
\end{cases}
\end{equation*}

\begin{lemma}\label{extensionlemma}
	Let $H$ be an open subgroup of a LCA group $G$, and let $\Omega_{+}$ and $\Omega_{-}$ be subsets of $G$. If a function $\varphi : H \to \mathbb{R}$ is in $\mathcal{F}_{H}(\Omega_{+} \cap H, \Omega_{-} \cap H)$, then its trivial extension $\widetilde{\varphi}: G \to \mathbb{R}$ is in $\mathcal{F}_{G}(\Omega_{+}, \Omega_{-})$.
\end{lemma}

For a proof of Lemma \ref{extensionlemma}, we refer the reader to a proof of an analogous statement \cite[Lemma~1]{ramabulana}. We shall also make use of the following lemma.

\begin{lemma}\label{restrictionlemma}
	Let $H$ be an open subgroup of a LCA group $G$, and let $\Omega_{+}$ and $\Omega_{-}$ be subsets of $G$ with $\Omega_{+} \subset H$. If a function $\varphi: G \to \mathbb{R}$ is in $\mathcal{F}_{G}(\Omega_{+}, \Omega_{-})$, then its restriction $\varphi |_{H}: H \to \mathbb{R}$ to the subgroup $H$ is in $\mathcal{F}_{H}(\Omega_{+}, \Omega_{-} \cap H)$.
\end{lemma}

\begin{proof}
	Suppose that $\varphi: G \to \mathbb{R}$ is a continuous positive definite function with $\varphi^{-1}(0, \infty) \subset \Omega_{+}$ and $\varphi ^{-1}(-\infty, 0) \subset \Omega_{-}$, and satisfying $\varphi(0)=1$. Then its restriction $\varphi|_{H}: H \to \mathbb{R}$ to $H$ is a continuous positive definite function with $\varphi|_{H}^{-1}(0, \infty) \subset \Omega_{+}$ and $\varphi|_{H}^{-1}(-\infty, 0) \subset \Omega_{-} \cap H$ , and satisfies $\varphi|_{H}^{-1}(0) = 1$.
\end{proof}

\begin{lemma}\label{sigma-compact-lemma}
	A Haar measurable subset $\Omega$ of a LCA group $G$ is $\sigma$-finite if and only if it is contained in an open $\sigma$-compact subgroup.
\end{lemma}

\begin{proof} It is obvious that if $\Omega$ is Haar measurable and is contained in an open $\sigma$-compact subgroup, then it is $\sigma$-finite. So, it remains to show that if $\Omega$ is $\sigma$-finite, then it is contained in an open $\sigma$-compact subgroup. To that end, let us write $\Omega$ as a union
\begin{equation*}
\Omega = \bigcup_{n=1}^{\infty}A_{n},
\end{equation*}
where $A_{n}$ has finite Haar measure for all $n \in \mathbb{N}$. By outer regularity of the Haar measure on $G$, for all $n \in \mathbb{N}$, there is an open set $U_{n}$ in $G$ containing $A_{n}$ with $\lambda_{G}(U_{n}) < \infty$. We claim that for each $n \in \mathbb{N}$, $U_{n}$ lies in an open $\sigma$-compact subgroup $H_{n}$. To see this, let $K$ be any $\sigma$-compact open subgroup of $G$ and note that $G$ is a disjoint union of cosets $g + K$ of $K$ with $g$ running through the set of coset representatives of $K$. For each $n \in \mathbb{N}$, observe that since the cosets of $K$ are disjoint and cover $G$, the collection of subsets $(g + K) \cap U_{n}$ of $U_{n}$ as $g$ runs through the set of coset representatives of $K$ are disjoint and their union is $U_{n}$. Now, being open sets, either $(g + K) \cap U_{n} = \emptyset$ or $\lambda_{G}((g + K) \cap U_{n}) >0$. However, $\lambda_{G}((g + K) \cap U_{n}) >0$, and hence $(g + K) \cap U_{n} \ne \emptyset$, for at most countably many cosets $g + K$ of $K$, otherwise a contradiction of the finitude of the Haar measure of $U_{n}$ would arise. By Lemma \ref{sigmacompactgeneration}, $U_{n}$ is contained in the open $\sigma$-compact subgroup generated by the union of $K$ and the countably many $\sigma$-compact cosets of $K$ that meet $U_{n}$ nontrivially. Denote this subgroup by $H_{n}$. We obtain a sequence $(H_{n})_{n \in \mathbb{N}}$ of open $\sigma$-compact subgroups of $G$ having $A_{n} \subset U_{n} \subset H_{n}$ for all $n \in \mathbb{N}$, to which we apply Lemma \ref{sigmacompactgeneration} once again to obtain an open $\sigma$-compact subgroup $H$ of $G$ generated by $\bigcup_{n=1}^{\infty}H_{n}$. Clearly, $H$ contains $\Omega$, so the lemma is proved.
	
\end{proof}

\begin{theorem}\label{reduction}
	Let $G$ be a LCA group, and $\Omega_{+}$ and $\Omega_{-}$ be symmetric subsets of $G$ with $\Omega_{+}$ a $\sigma$-finite neighbourhood of $0$. Then there is an open $\sigma$-compact subgroup $H$ of $G$ containing $\Omega_{+}$ such that
\begin{equation*}
\mathcal{C}_{G}(\Omega_{+}, \Omega_{-}) = \mathcal{C}_{H}(\Omega_{+}, \Omega_{-}\cap H).
\end{equation*}
\end{theorem}

\begin{proof} On account of Lemma \ref{sigma-compact-lemma}, let $H$ be an open $\sigma$-compact subgroup of $G$ containing $\Omega_{+}$. Equip $H$ with Haar measure $\lambda_{H} := \lambda_{G}|_{H}$, the restriction of $\lambda_{G}$ to $H$.

Let $(f_{n})_{n \in \mathbb{N}}$ in $\mathcal{F}_{G}(\Omega_{+}, \Omega_{-})$ be a $\mathcal{C}_{G}(\Omega_{+}, \Omega_{-})$-extremal sequence chosen in such a way that
	\begin{equation*}\label{extremalGsequence}
	\int_{G}f_{n}\mbox{d}\lambda_{G} \ge \mathcal{C}_{G}(\Omega_{+}, \Omega_{-}) - \frac{1}{n} \textup{ for all } n \in \mathbb{N}.
	\end{equation*}
Let $f_{n}|_{H}: H \to \mathbb{R}$ denote the restriction of $f_{n}$ to $H$. By Lemma \ref{restrictionlemma}, $f_{n}|_{H}$ is in $\mathcal{F}_{G}(\Omega_{+}, \Omega_{-}\cap H)$ for all $n \in \mathbb{N}$. For all $n \in \mathbb{N}$, we have
	\begin{equation}\label{splitintegral}
	\int_{G}f_{n}\mbox{d}\lambda_{G} = \int_{H}f_{n}\mbox{d}\lambda_{G} + \int_{G\symbol{92}H}f_{n}\mbox{d}\lambda_{G}.
	\end{equation}
Since $f_{n}^{-1}(0, \infty)$ is contained in $H$, we have that $f_{n}(g) \le 0$ for all $g \in G \symbol{92}H$. Therefore,
	\begin{equation}\label{negint}
	\int_{G\symbol{92}H}f_{n}\mbox{d}\lambda_{G} \le 0.
	\end{equation}
From (\ref{splitintegral}) and (\ref{negint}), it follows that
	\begin{equation}\label{ineq}
	\int_{H}f_{n}|_{H}\mbox{d}\lambda_{H} = \int_{H}f_{n}\mbox{d}\lambda_{G} \ge \int_{G}f_{n}\mbox{d}\lambda _{G}.
	\end{equation}
Now, by the definition of $\mathcal{C}_{H}(\Omega_{+}, \Omega_{-} \cap H)$ and inequality (\ref{ineq}), we have
	\begin{equation*}
	\mathcal{C}_{H}(\Omega_{+}, \Omega_{-} \cap H) \ge \int_{H}f_{n}|_{H}\textup{d}\lambda_{H} \ge \int_{G}f_{n}\textup{d}\lambda_{G} \ge \mathcal{C}_{G}(\Omega_{+}, \Omega_{-}) - \frac{1}{n}.
	\end{equation*}
Thus $\mathcal{C}_{H}(\Omega_{+}, \Omega_{-} \cap H) \ge \mathcal{C}_{G}(\Omega_{+}, \Omega_{-})$. For the other inequality, let $(h_{n})_{n \in \mathbb{N}}$ in  $\mathcal{F}_{H}(\Omega_{+}, \Omega_{-} \cap H)$ be a $\mathcal{C}_{H}(\Omega_{+}, \Omega_{-} \cap H)$-extremal sequence such that
	\begin{equation*}
	\int_{H}h_{n}\mbox{d}\lambda_{H} \ge \mathcal{C}_{H}(\Omega_{+} , \Omega_{-} \cap H)-\frac{1}{n} \textup{ for all } n \in \mathbb{N}.
	\end{equation*}
By Lemma \ref{extensionlemma}, the sequence $(\widetilde{h_{n}})_{n \in \mathbb{N}}$ of trivial extensions is in $\mathcal{F}_{G}(\Omega_{+}, \Omega_{-} \cap H) \subset ~ \mathcal{F}_{G}(\Omega_{+}, \Omega_{-})$, and since $\widetilde{h_{n}}$ vanishes outside of $H$, we have that for all $n \in \mathbb{N}$,
	\begin{equation*}
	\mathcal{C}_{G}(\Omega_{+}, \Omega_{-}) \ge \int_{G}\widetilde{h_{n}}\textup{d}\lambda_{G} = \int_{H}h_{n}\textup{d}\lambda_{H} \ge \mathcal{C}_{H}(\Omega_{+}, \Omega_{-} \cap H) - \frac{1}{n}.
	\end{equation*}
Thus, $\mathcal{C}_{G}(\Omega_{+}, \Omega_{-}) \ge \mathcal{C}_{H}(\Omega_{+}, \Omega_{-} \cap H)$, and hence
	\begin{equation*}
	\mathcal{C}_{G}(\Omega_{+}, \Omega_{-}) = \mathcal{C}_{H}(\Omega_{+}, \Omega_{-} \cap H),
	\end{equation*}
as required.

\end{proof}

As a corollary of Theorem \ref{reduction}, we obtain that the Delsarte and Tur\'{a}n problems on any LCA group may be reduced to some open $\sigma$-compact subgroup.

\begin{corollary}[Delsarte problem]
	Let $G$ be a LCA group and $\Omega \subset G$ a $\sigma$-finite symmetric neighbourhood of $0$. Then there is an open $\sigma$-compact subgroup $H$ of $G$ containing $\Omega$ such that
	\begin{equation*}
	\mathcal{D}_{G}(\Omega) = \mathcal{D}_{H}(\Omega).
	\end{equation*}
\end{corollary}

\begin{corollary}[Tur\'{a}n problem]
	Let $G$ be a LCA group and $\Omega \subset G$ a $\sigma$-finite symmetric neighbourhood of $0$. Then there is an open $\sigma$-compact subgroup $H$ of $G$ containing $\Omega$ such that
	\begin{equation*}
	\mathcal{T}_{G}(\Omega) = \mathcal{T}_{H}(\Omega).
	\end{equation*}
\end{corollary}

In Theorem \ref{reduction} it is not assumed that $\Omega_{-}$ is $\sigma$-finite. So, $\Omega_{-}$ may not lie in an open $\sigma$-compact subgroup of $G$. This is the reason for taking the intersection $\Omega_{-} \cap H$ in Theorem \ref{reduction}. However, if both $\Omega_{+}$ and $\Omega_{-}$ are $\sigma$-finite, then so is their union $\Omega_{+} \cup \Omega_{-}$. Hence the latter is contained in an open $\sigma$-compact subgroup by Lemma \ref{sigma-compact-lemma}. In this case, it is easy to modify the proof of Theorem \ref{reduction}, by taking $H$ to contain not just $\Omega_{+}$ but the union $\Omega_{+} \cup \Omega_{-}$ to obtain the following theorem.

\begin{theorem}\label{reduction2}
	Let $G$ be a LCA group, and $\Omega_{+}$ and $\Omega_{-}$ be $\sigma$-finite symmetric subsets of $G$ with $\Omega_{+}$ a neighbourhood of $0$. Then  there is an open $\sigma$-compact subgroup $H$ of $G$ containing $\Omega_{+}$ and $\Omega_{-}$ such that
	
	\begin{equation*}
	\mathcal{C}_{G}(\Omega_{+}, \Omega_{-}) = \mathcal{C}_{H}(\Omega_{+}, \Omega_{-}).
	\end{equation*}
\end{theorem}

The next theorem shows that for the sake of computing the extremal values, we may as well only consider $\sigma$-compact groups, without any assuming that $\Omega_{+}$ and $\Omega_{-}$ are $\sigma$-compact. This is made possible by the fact that the support of a Haar integrable function can not be too large, in the sense expressed by the following lemma.

\begin{lemma}[{\cite[Corollary 1.3.6]{deitmar}}]\label{supportlemma} Let $G$ be a LCA group. For any $f \in L^1(G)$ the support of $f$ is contained in an open $\sigma$-compact subgroup.
\end{lemma}

\begin{theorem} Let $G$ be a LCA group, and $\Omega_{+}$ and $\Omega_{-}$ be symmetric subsets of $G$ with $\Omega_{+}$ a neighbourhood of $0$. Then there is an open $\sigma$-compact subgroup $H$ of $G$ such that
	\begin{equation*}
	\mathcal{C}_{G}(\Omega_{+}, \Omega_{-}) = \mathcal{C}_{H}(\Omega_{+} \cap H, \Omega_{-} \cap H).
	\end{equation*}
	
\end{theorem}

\begin{proof}
	
Using Lemma \ref{extensionlemma}, the inequality $\mathcal{C}_{H}(\Omega_{+} \cap H, \Omega_{-} \cap H) \le \mathcal{C}_{G}(\Omega_{+}, \Omega_{-})$ is immediate. So, it remains to show the reverse inequality. Let $(f_{n})_{n \in \mathbb{N}}$ in $\mathcal{F}_{G}(\Omega_{+}, \Omega_{-})$ be a $\mathcal{C}_{G}(\Omega_{+}, \Omega_{-})$-extremal sequence chosen in such a way that
	\begin{equation*}
	\int_{G}f_{n}\mbox{d}\lambda_{G} \ge \mathcal{C}_{G}(\Omega_{+}, \Omega_{-}) - \frac{1}{n} \textup{ for all } n \in \mathbb{N}.
	\end{equation*}
By Lemma \ref{supportlemma}, for each $n \in \mathbb{N}$, there is an open $\sigma$-compact subgroup $H_{n}$ of $G$ such that $\supp f_{n} \subset H_{n}$. Let $H$ be the subgroup of $G$ generated by the union of all the $H_{n}$'s as $n$ runs through $\mathbb{N}$. By Lemma \ref{sigmacompactgeneration}, $H$ is an open $\sigma$-compact subgroup of $G$. Equip $H$ with Haar measure $\lambda_{H} := \lambda_{G}|_{H}$.
	
Observe that for all $n \in \mathbb{N}$, $f_{n}(g) = 0$ for all $g \in G \symbol{92} H$. Let $f_{n}|_{H}: H \to \mathbb{R}$ be the restriction of $f_{n}$ to the subgroup $H$. By Lemma~\ref{restrictionlemma} $f_{n}|_{H} \in \mathcal{F}_{H}(\Omega_{+} \cap H, \Omega_{-} \cap H)$. We have
	\begin{equation*}
	\mathcal{C}_{H}(\Omega_{+} \cap H, \Omega_{-} \cap H) \ge \int_{H}f_{n}|_{H}\mbox{d}\lambda_{H} = \int_{G}f_{n}\mbox{d}\lambda_{G} \ge \mathcal{C}_{G}(\Omega_{+}, \Omega_{-}) - \frac{1}{n}.
	\end{equation*}
Taking limits furnishes  $\mathcal{C}_{H}(\Omega_{+} \cap H, \Omega_{-} \cap H) =  \mathcal{C}_{G}(\Omega_{+}, \Omega_{-}).$
	
\end{proof}

\section{Existence of Extremal Functions}

In this section, we show the existence of an extremal function for the extremal problem \eqref{extremalproblem} for boundary-coherent sets. We first obtain existence in the context of $\sigma$-compact groups and later extend it to the general case. We start with a few lemmas.

\begin{lemma}[Approximation of Unity, {\cite[Lemma 2]{kolrev}}] \label{approx_unity}

Let $C$ be an arbitrary compact subset of a LCA group $G$ and $\ee > 0$ be an arbitrary real number. Then there exists a continuous compactly supported positive definite function $k: G \to \mathbb{R}$ such that $k(0)=1$, $0\le k \le 1$, and $k|_{C} > 1-\ee$.
	
\end{lemma}

\begin{lemma}[Mazur's lemma, {\cite[Corollary 3.8, Exercise 3.4]{brezis}}] \label{mazur} Let $E$ be a Banach space and let $(x_{n})_{n \in \mathbb{N}}$ be a sequence in $E$ converging to $x \in E$ weakly. Then there exists a sequence $(y_{n})_{n \in \mathbb{N}}$ in $E$ such that
	\begin{equation*}
	y_{n} \in \mathrm{conv}\left(\bigcup_{i=n}^{\infty}\{x_{i}\}\right) \mbox{ for all } n \in \mathbb{N}
	\end{equation*}
and $(y_n)_{n \in \mathbb{N}}$ converges strongly to $x$ in $E$, where for a subset $A \subset E$, $\mathrm{conv}(A)$ refers to the convex hull of $A$, and strong convergence refers to convergence with respect to the norm topology on $E$.
	
\end{lemma}

\begin{theorem}\label{sigmacompactcase}
	Let $G$ be a $\sigma$-compact LCA group, and $\Omega_{+}$ and $\Omega_{-}$ be boundary-coherent symmetric subsets of $G$ with $\Omega_{+}$ a neighbourhood of $0$ having finite Haar measure. Then there exists a $\mathcal{C}_{G}(\Omega_{+}, \Omega_{-})$-extremal function $f \in \mathcal{F}_{G}(\Omega_{+}, \Omega_{-})$ satisfying $\int_{G}f\textup{d}\lambda_{G} = \mathcal{C}_{G}(\Omega_{+}, \Omega_{-})$.
\end{theorem}

\begin{proof} We start with using only the standard assumptions and will invoke boundary-coherence only at the end. In fact, we believe that this condition might be unnecessary, but we were not able to construct a proof without it.

	Note that for all $f \in \mathcal{F}_{G}(\Omega_{+}, \Omega_{-}),|f(g)| \le 1$ for all $g \in G$. Furthermore, since $f$ is an integrable continuous positive definite function, the non-negativity of its integral implies that $\int_{G}f_{-}\mbox{d}\lambda_{G} \le \int_{G}f_{+}\mbox{d}\lambda_{G} \le \lambda_G(\Omega)$. Hence $\|f\|_{L^{1}(G)} \le 2 \lambda_{G}(\Omega)$, which in turn implies that $\|f\|_{L^{2}(G)} \le 2\lambda_{G}(\Omega)$ for all $f \in \mathcal{F}_{G}(\Omega_{+}, \Omega_{-})$. In other words, $\mathcal{F}_{G}(\Omega_{+}, \Omega_{-})$ is a bounded subset of $L^{2}(G)$, and as such it belongs to the strong closure $\overline{B}_{2\lambda_{G}(\Omega)}(0)$ of the ball $B_{2\lambda_{G}(\Omega)}(0)$ of radius $2\lambda_{G}(\Omega)$ in $L^2(G)$ centred at 0. Being a closed and bounded subset of $L^2(G)$, $\overline{B}_{2\lambda_{G}(\Omega)}(0)$ is weakly sequentially compact according to \cite[Theorem 3.18]{brezis}. Now let $(f_{n})_{n \in \mathbb{N}}$ be a $\mathcal{C}_{G}(\Omega_{+}, \Omega_{-})$-extremal sequence in $\mathcal{F}_{G}(\Omega_{+}, \Omega_{-})$, that is, $\lim_{n \to \infty} \int_{G}f_n\mbox{d}\lambda_{G} = \mathcal{C}_{G}(\Omega_{+}, \Omega_{-})$. By the weak sequential compactness of $\overline{B}_{2\lambda_{G}(\Omega)}(0)$, there is a subsequence of $(f_{n})_{n \in \mathbb{N}}$ that converges weakly in $L^2(G)$ to some $f \in \overline{B}_{2\lambda_{G}(\Omega)}(0)$. By replacing $(f_{n})_{n \in \mathbb{N}}$ with this subsequence, assume that $(f_{n})_{n \in \mathbb{N}}$ converges weakly in $L^2(G)$ to $f$. Since $\mathcal{F}_{G}(\Omega_{+}, \Omega_{-})$ is a convex subset of $L^2(G)$, Lemma \ref{mazur} implies that there is a sequence $(h_{n})_{n \in \mathbb{N}}$ in $\mathcal{F}_{G}(\Omega_{+}, \Omega_{-})$ chosen such that $h_{n} \in \mathrm{conv} (\bigcup _{i =n}^{\infty}{f_{i}})$ for all $n \in \mathbb{N}$, and converging to $f$ strongly in $L^2(G)$. We claim that the latter sequence is $\mathcal{C}_{G}(\Omega_{+}, \Omega_{-})$-extremal. First note that by the extremality of $(f_{n})_{n \in \mathbb{N}}$, we have that for any $\varepsilon > 0$ there exists $N \in \mathbb{N}$ such that for all $n > N$ it holds that
	\begin{equation*}
	\left |\int_{G}f_{n}\mbox{d}\lambda_{G} - \mathcal{C}_{G}(\Omega_{+}, \Omega_{-}) \right| < \varepsilon.
	\end{equation*}
Now, write $h_{n} = \sum_{i =n}^{\infty}c_{i}^{n}f_{i}$ where $c_{i}^{n} \ge 0$, $\sum_{i = n}^{\infty}c_{i}^{n} = 1$, and $c_{i}^{n} = 0$ for all but finitely many $i \ge n$. Then for any $n \ge N$, we have
	\begin{equation*}
	\left | \int_{G}h_{n}\mbox{d}\lambda_{G} - \mathcal{C}_{G}(\Omega_{+}, \Omega_{-}) \right| \le \sum_{i =n}^{\infty}c_{i}^{n}\left|\int_{G}f_{i}\mbox{d}\lambda_{G} - \mathcal{C}_{G}(\Omega_{+}, \Omega_{-}) \right| < \sum_{i = n}^\infty c_{i}^{n}\varepsilon = \varepsilon.
	\end{equation*}
Therefore $(h_{n})_{n \in \mathbb{N}}$ is $\mathcal{C}_{G}(\Omega_{+}, \Omega_{-})$-extremal. So, replacing, if necessary, the sequence $(f_{n})_{n \in \mathbb{N}}$ by this new sequence, assume that $(f_{n})_{n \in \mathbb{N}}$ converges to $f$ strongly in $L^2(G)$.

By passing to a pointwise almost everywhere convergent subsequence, assume that $(f_{n})_{n \in \mathbb{N}}$ converges to $f$ pointwise almost everywhere. The fact that for all $n \in \mathbb{N}$, $|f_{n}(g)| \le 1$ for all $g \in G$ and $(f_{n})_{n \in \mathbb{N}}$ converges almost everywhere to $f$ implies that $f$ is bounded in $L^{\infty}(G)$ by 1. Since $(f_{n})_{n \in \mathbb{N}}$ converges weakly in $L^2(G)$ to $f$, we have
	$$
	\int_{G}(\Star{h} \ast h)f \mbox{d}\lambda_{G} = \lim_{n \to \infty} \int_{G}(\Star{h}\ast h)f_{n}\mbox{d}\lambda_{G} \ge 0,
	$$
	for all compactly supported continuous complex-valued functions $h$. This shows that $f$ is integrally positive definite. The function $f$ being integrally positive definite and $G$ being $\sigma$-compact implies that $f$ agrees almost everywhere with a continuous positive definite function. Thus, by correcting $f$ on a set of Haar measure zero, thus not changing the value of its integral, assume that $f$ is a continuous positive definite function.
	
	Let $S_{+}:=f^{-1}(0,\infty)$ and $S_{-}:=f^{-1}(-\infty,0)$. If $f_n(x)\to f(x)>0$, then $x \in \Omega_{+}$, and similarly for negative values. Therefore, by the almost everywhere convergence, we have that $S_{+}\setminus\Omega_{+}$ and $S_{-}\setminus\Omega_{-}$ are sets of measure zero.

	Next we show that $f \in L^1(G)$. Since $f^{-1}(0, \infty) =S_{+}$ and $S_{+} \setminus \Omega_{+}$ is of measure zero, we have,
	\begin{equation}\label{positivitybound}
	\int_{G}f_{+}\textup{d}\lambda_{G} = \int_{S_{+}}f_{+}\textup{d}\lambda_{G} = \int_{\Omega_{+}} f_{+}\textup{d}\lambda_{G} \le \lambda_{G}(\Omega_{+}) < \infty.
	\end{equation}
	On the other hand, by Fatou's lemma,
	\begin{equation}\label{negativitybound}
	\begin{split}
	\int_{G}f_{-}\textup{d}\lambda_{G} & \le \liminf_{n \to \infty} \int_{G}(f_{n})_{-}\textup{d}\lambda_{G}\\
	&\le \liminf_{n \to \infty}\int_{G}(f_{n})_{+}\textup{d}\lambda_{G} \\
	& \le \lambda_{G}(\Omega_{+}) < \infty,
	\end{split}
	\end{equation}
	where the second inequality of \eqref{negativitybound} is obtained using that $f_{n}$ being positive definite and integrable implies that
	\begin{equation*}
	0 \le \int_{G}f_{n}\mbox{d}\lambda_{G} = \int_{G}(f_{n})_{+}\mbox{d}\lambda_{G} - \int_{G}(f_{n})_{-}\mbox{d}\lambda_{G}.
	\end{equation*}
	Together, \eqref{positivitybound} and \eqref{negativitybound} imply that
	\begin{equation}
	\int_{G}|f|\textup{d}\lambda_{G} =   \int_{G}f_{+}\textup{d}\lambda_{G} +  \int_{G}f_{-}\textup{d}\lambda_{G} < \infty.
	\end{equation}
	Hence $f \in L^1(G)$.
	
	Note that the sequence $((f_{n})_{+})_{n \in \mathbb{N}}$ of the positive parts of the $f_{n}$'s converges pointwise almost everywhere to $f_{+}$, and that $(f_{n})_{+} \le \mathbf{1}_{\Omega_{+}}$. Therefore, by Lebesgue's Dominated Convergence theorem and $\lambda_{G}(\Omega_{+}) < \infty$, we have
	\begin{equation}\label{eq10}
	\int_{G}f_{+}\mbox{d}\lambda_{G} =
	\lim_{n \to \infty} \int_{G}(f_{n})_{+}\mbox{d}\lambda_{G}.
	\end{equation}
	Recall that from  \eqref{negativitybound} we have the inequality
	\begin{equation}\label{eq9}
	\int_{G}f_{-} \lambda_{G}  \mbox{d}\lambda_{G} \le \liminf_{n \to \infty}\int_{G}(f_{n})_{-}\mbox{d}\lambda_{G}.
	\end{equation}
	Now, subtracting \eqref{eq9} from \eqref{eq10} we have
	\begin{equation*}
	\begin{split}
	\int_G f  \text{d}\lambda_G &\ge  \lim_{n \to \infty} \int_{G}(f_n)_{+}\text{d}\lambda_G - \liminf_{n \to \infty}{\int_{G}(f_n)_{-}} \text{d}\lambda_G \\
	&= \limsup_{n \to \infty}{\int_{G}\left((f_n)_{+} - (f_n)_{-}\right) \text{d}\lambda_G} \\
	&= \lim_{n \to \infty}\int_{G}f_n \mbox{d}\lambda _{G} = \mathcal{C}_{G}(\Omega_{+}, \Omega_{-}).
	\end{split}
	\end{equation*}
	
	Only now we employ boundary-coherence to show that $f^{-1}(0,\infty) \subset \Omega_{+}$ and also $f^{-1}(-\infty,0) \subset \Omega_{-}$. First, if $U$ is a non-empty open subset of $\ext\Omega_{+}$, then almost everywhere on $U$ we have $f_n\to f$, while $f_n \le 0$, hence also $f \le 0$. Consider an arbitrary $x \not \in \Omega_{+}$. By continuity of $f$, fixing any $\ve>0$ we can find a sufficiently small open neighbourhood $V$ of $x$ such that $f(y) \ge f(x)- \ve$ for $y \in V$. Whether $x \in \ext \Omega_{+}$ or is only in $\partial \Omega_{+}$, by boundary-coherence there are points $z \in V \cap \ext \Omega_{+}$, so in particular there is a small neighbourhood $U\subset V$ of $z$ fully in $\ext \Omega_{+}$. By the above we know that $f\le 0$ almost everywhere on $U$. So, picking a point $w\in U$ satisfying $f(w) \le 0$, we also have $w \in V$, and hence $0 \ge f(w) \ge f(x)-\ve$. It follows that $f(x)\le \ve$, and as this holds for all $\ve>0$, we find $f(x)\le 0$. So, if $x \not\in \Omega_{+}$, then $x \not\in S_{+}$ either. Therefore, $S_{+}\subset \Omega_{+}$. Note that we could deduce this only by an application of the extra topological condition of boundary-coherence. The same argument gives also $S_{-} \subset \Omega_{-}$.
	
	To show that $f \in \mathcal{F}(\Omega_{+}, \Omega_{-})$, it remains to show that $f(0) = 1$. The fact that $\int_{G}f\mbox{d}\lambda_{G} \ge  \mathcal{C}_{G}(\Omega_{+},\Omega_{-}) > 0$, implies that $f$ is not identically zero, and in particular, by positive definiteness, implies that $f(0) \ne 0$. Suppose for a contradiction that we have $f(0) \ne 1$. Recall that $\|f\|_{L^\infty(G)} \le 1$. So, $0 < f(0) < 1$. Then the function $h = f/f(0)$ is in $\mathcal{F}_{G}(\Omega_{+}, \Omega_{-})$ and has integral greater than $\mathcal{C}_{G}(\Omega_{+}, \Omega_{-})$, contradicting the extremality of $\mathcal{C}_{G}(\Omega_{+}, \Omega_{-})$. It follows that $f(0) = 1$ and hence $f \in \mathcal{F}_{G}(\Omega_{+}, \Omega_{-})$. This also shows that $f$ is extremal.
	
\end{proof}

\begin{proposition}

For the extremal sequence $(f_n)_{n \in \mathbb{N}}$ in Theorem~\ref{sigmacompactcase} that converges to  $f$ strongly in $L^2(G)$ and pointwise almost everywhere, we have, in addition that $f_n \to f$ in $L^1(G)$.

\end{proposition}

\begin{proof}
Take any $\ee>0$. As $f \in L^1(G)$, there is a compact set $K\Subset G$ with $\int_{G\setminus K} |f|\textup{d}\lambda_{G} <\ee$. Also, as $\Omega_{+}$ and hence $S_{+}$ as well as $S_{+}\cup \Omega_{+}$ have finite measure, by inner regularity of the Haar measure there exists a compact set $K'$ such that  $\lambda_G(( S_{+}\cup \Omega_{+} ) \setminus K') <\ee$. We may assume that $K \supset K'$ and hence $K$ satisfies both requirements. By Lemma \ref{approx_unity}, we can find a function $k$, which  is continuous, positive definite, compactly supported, $0\le k\le 1$, and $k|_K >  1-\ee$.
	
	We have
	\begin{equation*}
	\begin{split}
	\int_G |f|(1-k)\textup{d}\lambda_{G} & \le \int_K |f|(1-k)\textup{d}\lambda_{G} + \int_{G \setminus K}  |f|(1-k)\textup{d}\lambda_{G} \\
	&\le \eta \|f\|_{L^1(G)} +  \int_{G \setminus K} |f| \textup{d}\lambda_{G} \\
	&\le \eta (\|f\|_{L^1(G)} + 1) \le \eta (2 \lambda_G(\Omega_+) + 1).
	\end{split}
	\end{equation*}
	
	Next, we refer to the $L^2$ convergence $f_n \to f$. As $k\in L^2(G)$, we have $\int_G |f-f_n|k \textup{d}\lambda_{G}\to 0$, and thus $\int_G |f-f_n|k \textup{d}\lambda_{G} < \eta$ for $n$ large enough.
	
	Let us assume without loss of generality that $\int_G f_n\textup{d}\lambda_{G} \ge \mathcal{C}_{G}(\Omega_{+},\Omega_{-}) - \frac{1}{n}$ for all $n \in \mathbb{N}$. Taking into account $f_n k\in \mathcal{F}_{G}(\Omega_{+},\Omega_{-})$\footnote{We use here the fact that the product of two positive definite functions is positive definite.}, we necessarily have $\int_G f_nk \textup{d}\lambda_{G} \le \mathcal{C}_{G}(\Omega_{+},\Omega_{-})$. Thus we can write
	\begin{equation*}
	\begin{split}
	\mathcal{C}_{G}(\Omega_{+},\Omega_{-}) \ge \int_G f_nk \textup{d}\lambda_{G} & = \int_G f_n \textup{d}\lambda_{G} - \int_G f_n(1-k)\textup{d}\lambda_{G} \\
	& \ge \mathcal{C}_{G}(\Omega_{+},\Omega_{-}) - \frac{1}{n} - \int_G (f_n)_{+}(1-k) \textup{d}\lambda_{G}+\int_G (f_n)_{-}(1-k)\textup{d}\lambda_{G},
	\end{split}
	\end{equation*}
therefore
	\begin{equation*}
	\int_G (f_n)_{-}(1-k) \textup{d}\lambda_{G}\le \int_G (f_n)_{+}(1-k)\textup{d}\lambda_{G} + \frac{1}{n}.
	\end{equation*}
	We get
	\begin{equation}\label{remark_L1_est1}
	\int_G (1-k)|f_n|\textup{d}\lambda_{G} =\int_G (1-k)(f_n)_{+}\textup{d}\lambda_{G} +\int_G (1-k)(f_n)_{-}\textup{d}\lambda_{G} \le 2 \int_G (f_n)_{+}(1-k) \textup{d}\lambda_{G} + \frac{1}{n}.
	\end{equation}
	Using the facts that $0 \le 1 - k < \eta$ on $K$, $0 \le k \le 1$ on $G$, $(f_n)_+ \le 1$ on $G$, and $(f_n)_+$ vanishes outside $\Omega_+$, we estimate
	\begin{equation*}
	\begin{split}
	\int_G (f_n)_{+}(1-k) \textup{d}\lambda_{G}
	& = \int_K (f_n)_{+}(1-k) \textup{d}\lambda_{G} + \int_{G \setminus K} (f_n)_{+}(1-k) \textup{d}\lambda_{G} \\
	& < \eta \int_K (f_n)_{+} \textup{d}\lambda_{G} + \int_{\Omega_+ \setminus K} (f_n)_{+}(1-k) \textup{d}\lambda_{G} \\
	& \le \eta \int_{\Omega_+} (f_n)_{+} \textup{d}\lambda_{G} + \lambda_G( \Omega_+ \setminus K ) \\
	& \le \eta \lambda_G(\Omega_+) + \eta = \eta( \lambda_G(\Omega_+) + 1).
	\end{split}
	\end{equation*}
	With this estimate at hand, \eqref{remark_L1_est1} gives
	\begin{equation*}
	\int_G (1-k)|f_n|\textup{d}\lambda_{G} \le 2 \eta( \lambda_G(\Omega_+) + 1) + \frac{1}{n}
	\le  2 \eta( \lambda_G(\Omega_+) + 2)
	\end{equation*}
	for sufficiently large $n$. Finally,
	\begin{equation*}
	\begin{split}
	\int_G |f-f_n|\textup{d}\lambda_{G}
	& = \int_G |f-f_n|k\textup{d}\lambda_{G} + \int_G (1-k)|f-f_n|\textup{d}\lambda_{G} \\
	& \le \int_G |f-f_n|k\textup{d}\lambda_{G} + \int_G |f|(1-k)\textup{d}\lambda_{G} + \int_G |f_n|(1-k)\textup{d}\lambda_{G} \\
	& \le  \eta + \eta (2 \lambda_G(\Omega_+) + 1) + 2 \eta( \lambda_G(\Omega_+) + 2)
	\end{split}
	\end{equation*}
for $n$ large enough, and hence $f_n\to f$ in $L^1(G)$.
\end{proof}

Next, we show the existence of an extremal function for general LCA groups. The key is to reduce the problem to the case of $\sigma$-compact groups, where an extremal function is known to exist, and then extending the solution to the general case. As such, we first show in the following lemma that boundary-coherence is hereditary with respect to open subgroups.

\begin{lemma}\label{hereditary_lemma}
Let $H$ be an open subgroup of an LCA group $G$ and $\Omega$ a subset of $G$. If $\Omega$ is boundary-coherent in $G$ then $\Omega \cap H $ is boundary-coherent in $H$, i.e., $\partial_{H}(\Omega \cap H) \subset \overline{\ext_{H}(\Omega \cap H)}^{H}$.
\end{lemma}

\begin{proof}
Let us first notice that $\ext_H(\Omega \cap H) = \ext_G(\Omega) \cap H$. Indeed, the exterior of a set is the interior of its complement. Now the statement follows from an easy observation that for a set $A \subset G$ we have $\Int_H(A \cap H) = \Int_G(A) \cap H$ on account of $H$ being open.

Furthermore, since $H$ is an open subgroup it is closed in $G$, and hence  $\overline{\ext_H(\Omega \cap H)}^{H} = \overline{\ext_G(\Omega) \cap H}^{G}$. Moreover, the fact that $H$ is open implies that $\overline{\ext_G(\Omega) \cap H}^{G} = \overline{\ext_G(\Omega)}^{G} \cap H$. Indeed, the containment $\overline{\ext_G(\Omega) \cap H}^{G} \subset \overline{\ext_G(\Omega)}^{G} \cap H$  follows from the fact that $\ext_G(\Omega) \cap H$ is contained in $\overline{\ext_G(\Omega)}^{G} \cap H$, and the latter set is closed in $G$. For the reverse containment, let $x \in \overline{\ext_G(\Omega)}^{G} \cap H$. Take an arbitrary neighbourhood $V$ of $x$ in $G$. Since $H$ is open, $V \cap H$ is also a neighbourhood of $x$ in $G$.  Since $x \in \overline{\ext_G(\Omega)}^{G}$, we have that $V \cap H \cap \ext_{G}(\Omega) \ne \emptyset$. This implies that $x \in \overline{\ext_G(\Omega) \cap H}^{G}$.

Now, using the fact that $\partial_G(H) = \emptyset$, we conclude
\begin{equation*}
\begin{split}
	\partial_{H}(\Omega \cap H) &= \overline{\Omega \cap H}^{H} \cap \overline{{H}\setminus(\Omega \cap H)}^{H}\\
	&= \overline{\Omega \cap H}^{G} \cap \overline{{H}\setminus(\Omega \cap H)}^{G} \\
	&= \partial_{G}(\Omega \cap H) \cap H  \\
	& \subset (\partial_{G}(\Omega)\cup \partial_{G}(H)) \cap H \\
	& = \partial_{G}(\Omega) \cap H \\
	& \subset \overline{\ext_G(\Omega)}^{G} \cap H\\
	& = \overline{\ext_H(\Omega \cap H)}^{H}.
\end{split}
\end{equation*}
\end{proof}

\begin{theorem}\label{generalcase}
	Let $G$ be a LCA group, and $\Omega_{+}$ and $\Omega_{-}$ be boundary-coherent symmetric subsets of $G$ with $\Omega_{+}$ a neighbourhood of $0$ having finite Haar measure. Then there exists a $\mathcal{C}_{G}(\Omega_{+}, \Omega_{-})$-extremal function $f \in \mathcal{F}_{G}(\Omega_{+}, \Omega_{-})$ satisfying $\int_{G}f\textup{d}\lambda_{G} = \mathcal{C}_{G}(\Omega_{+}, \Omega_{-})$.
\end{theorem}

\begin{proof}
	On account of Theorem \ref{reduction}, let $H \supset \Omega_{+}$ be an open $\sigma$-compact subgroup of $G$ such that $\mathcal{C}_{G}(\Omega_{+}, \Omega_{-}) = \mathcal{C}_{H}(\Omega_{+}, \Omega_{-} \cap H)$. By Lemma \ref{hereditary_lemma}, $\Omega_{+} = \Omega_{+} \cap H$ and $\Omega_{-} \cap H$ are boundary-coherent. Therefore, by Theorem \ref{sigmacompactcase}, there exists a $\mathcal{C}_{H}(\Omega_{+}, \Omega_{-} \cap H)$-extremal function $f \in \mathcal{F}_{H}(\Omega_{+}, \Omega_{-} \cap H)$ such that
	
	\begin{equation}
	\int_{H}f\mbox{d}\lambda_{H} = \mathcal{C}_{H}(\Omega_{+}, \Omega_{-} \cap H).
	\end{equation}

	Let $\widetilde{f} : G \to \mathbb{R}$ be the trivial extension of $f$. Then $\widetilde{f}$ is in $\mathcal{F}_{G}(\Omega_{+}, \Omega_{-})$. We claim that $\widetilde{f}$ is a $\mathcal{C}_{G}(\Omega_{+}, \Omega_{-})$-extremal function. Indeed,
	
	\begin{equation*}
	\int_{G}\widetilde{f}\textup{d}\lambda_{G} = \int_{H}f\textup{d}\lambda_{H} = \mathcal{C}_{H}(\Omega_{+}, \Omega_{-} \cap H) = \mathcal{C}_{G}(\Omega_{+}, \Omega_{-}).
	\end{equation*}
\end{proof}

\begin{proposition}

Let $f$ be a $\mathcal{C}_{G}(\Omega_{+}, \Omega_{-})$-extremal function. Then there exists a sequence $(f_n)_{n \in \mathbb{N}}$ in $\mathcal{F}_{G}(\Omega_{+}, \Omega_{-})$ such that each $f_n$ has compact support and $(f_n)_{n \in \mathbb{N}}$ converges to $f$ in $L^1(G)$ and uniformly on compact sets.

\end{proposition}

\begin{proof}

Since $f \in L^1(G)$ and therefore $\supp{f}$ is $\sigma$-compact (see Lemma~\ref{supportlemma}), we can find a sequence of sets $(K_n)_{n \in \mathbb{N}}$ such that $\int_{G \setminus K_n} |f| \textup{d} \lambda_G < \frac{1}{n}$, $K_n \Subset G$, $K_n \subset K_{n+1}$, and $\supp{f} \subset \bigcup_{n \in \mathbb{N}} K_n$. By Lemma~\ref{approx_unity} there exist continuous, compactly supported, positive definite functions $k_n$ such that $k_n(0) = 1$, $0 \le k_n \le 1$, $\left. k_n \right|_{K_n} > 1 - \frac{1}{n}$.

Put $f_n := f k_n$, $n \in \mathbb{N}$. Clearly, $f_n \in \mathcal{F}_{G}(\Omega_{+}, \Omega_{-})$ and compactly supported.

To establish convergence in $L^1(G)$, we estimate
\begin{equation*}
\begin{split}
\int_{G} |f_n - f| \textup{d} \lambda_G
& \le \int_{K_n} |f_n - f| \textup{d} \lambda_G + \int_{G \setminus K_n} |f_n| \textup{d} \lambda_G + \int_{G \setminus K_n} |f| \textup{d} \lambda_G \\
& =  \int_{K_n} |f| (1 - k_n) \textup{d} \lambda_G + \int_{G \setminus K_n} |f| k_n \textup{d} \lambda_G + \int_{G \setminus K_n} |f| \textup{d} \lambda_G \\
& \le \frac{1}{n} \int_{K_n} |f| \textup{d} \lambda_G + 2 \int_{G \setminus K_n} |f| \textup{d} \lambda_G \\
& \le \frac{1}{n} (\|f\|_{L^1(G)} + 2).
\end{split}
\end{equation*}

Now let $C$ be any compact set in $G$. Then there exists $N \in \mathbb{N}$ such that $C \cap \supp{f} \subset K_n$ for all $n \ge N$. We have
\begin{equation*}
\|f_n - f\|_{L^\infty(C)} = \|f (1 - k_n)\|_{L^\infty(C)} = \|f (1 - k_n)\|_{L^\infty( C \cap \supp{f})}
\le \|f (1 - k_n)\|_{L^\infty(K_n)} \le \frac{1}{n},
\end{equation*}
which implies uniform convergence on $C$.

\end{proof}

\begin{corollary}[Tur\'{a}n problem]
	Let $G$ be a LCA group and $\Omega \subset G$ a boundary-coherent symmetric neighbourhood of $0$ having finite Haar measure. Then there exists an extremal function for the Tur\'{a}n constant $\mathcal{T}_{G}(\Omega)$.
\end{corollary}

\begin{corollary}[Delsarte problem]\label{Delsarteexistence}
	Let $G$ be a LCA group and $\Omega \subset G$ a boundary-coherent symmetric neighbourhood of $0$ having finite Haar measure. Then there exists an extremal function for the Delsarte constant $\mathcal{D}_{G}(\Omega)$.
\end{corollary}

\begin{example}
	
Consider the Tur\'an problem for $\Omega = (-2, -1) \cup (-1, 1) \cup (1, 2)$. Observe that $\Omega$ is not boundary-coherent. By \cite[Theorem 7]{kolrev}, the extremal constant is $\mathcal{T}_{\mathbb{R}}(\Omega) = 1$, which is attained by the function $f = \mathbf{1} _{\left [-\frac{1}{2},\frac{1}{2}\right]} \ast \mathbf{1} _{\left [-\frac{1}{2},\frac{1}{2}\right]}$ appearing on the left of Figure \ref{extfunctions1}. On the right of Figure \ref{extfunctions1} is the integrally positive definite function which is equal to $g = \frac{1}{2}\left(\mathbf{1}_{\left[-1,1\right]} \ast \mathbf{1}_{\left[-1,1\right]}\right)$ everywhere but the points $x = \pm 1$ where it vanishes, showing that allowing integrally positive definite (discontinuous) functions in our class can drastically increase the value of the extremal constant. Moreover, the function $f$ belongs to $\mathcal{F}_{\mathbb{R}}(\Omega, \Omega)$, showing that an extremal function can exist even when $\Omega = \Omega_{+}$ is not boundary-coherent.
	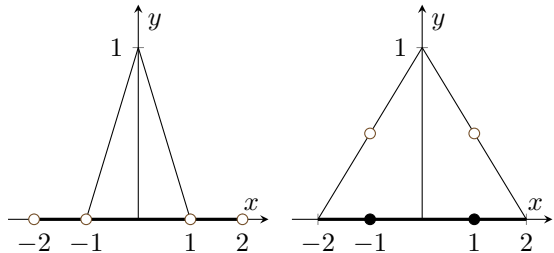
\begin{figure}[h]
		\begin{center}
		\begin{tikzpicture}
		\begin{axis}[scale = 0.5, ymajorticks=true, ytick = {1}, xtick={-2,-1,1,2}, clip = false, xmin= -2.5, ymin = 0, xmax = 2.5, ymax = 1.25, axis lines=middle, xlabel = $x$, ylabel = $y$
		]
		\addplot[domain = 0:1]{(1-x)};
		\addplot[domain = -1:0]{(1+x)};
\draw[very thick] (-2,0) -- (2,0);
		\addplot+[only marks,mark=*,mark options={scale=1, fill=white},text mark as node=true] coordinates {
			(-2,0) (-1,0) (1,0) (2,0)};
		\end{axis}
		\end{tikzpicture}
		\hskip 5pt
		\begin{tikzpicture}
		\begin{axis}[scale = 0.5, ymajorticks=true, ytick = {1}, xtick={-2,-1,1,2}, clip = false, xmin= -2.5, ymin = 0, xmax = 2.5, ymax = 1.25, axis lines=middle, xlabel = $x$, ylabel = $y$
		]
		\addplot[domain = 0:2]{(1-x/2)};
		\addplot[domain = -2:0]{(1+x/2)};
		\draw[very thick] (-2,0) -- (2,0);
		\addplot+[only marks,mark=*,mark options={scale=1, fill=white},text mark as node=true] coordinates {
			(-1,1/2)
			(1,1/2)};
		\addplot+[only marks,mark=*,mark options={scale=1, fill=black},text mark as node=true] coordinates {
			(-1,0)
			(1,0)};
		\end{axis}
		\end{tikzpicture}
		\end{center}
		\caption{Extremal functions for our class and if we allow functions which may be only integrally positive definite.}
		\label{extfunctions1}
	\end{figure}

\end{example}

\section{Equivalence of Extremal Problems}

Let $\Omega_{+}$ and  $\Omega_{-}$ be subsets of a LCA group $G$. In this section, we investigate conditions under which the extremal problem of calculating the extremal value $\mathcal{C}_{G}(\Omega_{+}, \Omega_{-})$ is equivalent to that of calculating the extremal value $\mathcal{C}^{*}_{G}(\Omega_{+}, \Omega_{-})$. This could either mean that the function classes $\mathcal{F}^{*}_{G}(\Omega_{+}, \Omega_{-})$ and $\mathcal{F}_{G}(\Omega_{+}, \Omega_{-})$ are equal, in which case the corresponding extremal values are equal, or it could mean that the extremal values are equal even though the corresponding function classes are not equal. For brevity, we drop the subscript $G$ and simply write, for example, $\mathcal{F}(\Omega_{+}, \Omega_{-})$ instead of $\mathcal{F}_{G}(\Omega_{+},  \Omega_{-})$.

Even though it is true that $\mathcal{C}(\Omega_{+}, \Omega_{-}) = \mathcal{C}^{*}(\Omega_{+}, \Omega_{-})$ for some classes of sets $\Omega_{+}, \Omega_{-}$, this can not always be proved by showing that $\mathcal{F}(\Omega_{+}, \Omega_{-}) = ~ \mathcal{F}^{*}(\Omega_{+}, \Omega_{-})$. This is because even though $\mathcal{F}^{*}(\Omega_{+}, \Omega_{-})\subset \mathcal{F}(\Omega_{+}, \Omega_{-})$, it is not true, in general, that $\mathcal{F}(\Omega_{+}, \Omega_{-}) \subset \mathcal{F}^{*}(\Omega_{+}, \Omega_{-})$, as the following example shows.

\begin{example}\label{turanexample}
	Let $G = \mathbb{R}$, $\Omega = \Omega_{+} = \Omega_{-} = (-1,1)$, so that we are in the situation of the Tur\'{a}n problem for an interval in $\mathbb{R}$. The extremal function for $\mathcal{T}_{\mathbb{R}}(\Omega)$ is the triangle function ${\mathbf{1}}_{\frac{1}{2}\Omega} \ast {\mathbf{1}}_{\frac{1}{2}\Omega} \in \mathcal{F}_{\mathbb{R}}(\Omega_{+}, \Omega_{-})$. It is easy to see that $\mathcal{T}_{\mathbb{R}}(\Omega) = \mathcal{T}^{*}_{\mathbb{R}}(\Omega) = 1$. However, the support of ${\mathbf{1}}_{\frac{1}{2}\Omega} \ast {\mathbf{1}}_{\frac{1}{2}\Omega}$ is $\left[-1,1\right]$, so that ${\mathbf{1}}_{\frac{1}{2}\Omega} \ast ~ {\mathbf{1}}_{\frac{1}{2}\Omega} \notin ~ \mathcal{F}^{*}_{\mathbb{R}}(\Omega_{+}, \Omega_{-})$.
\end{example}

In the opposite direction we have the following example, although it is quite special.

\begin{example}
	For a prime $p$, consider the group $\mathbb{Q}_{p}$ of $p$-adic rationals with the $p$-adic norm. Any ball of positive radius is both open and closed. So, if $\Omega_{+}$ and $\Omega_{-}$ are balls of positive radii with $\Omega_{+}$ centred at $0$, then the extremal problems are equivalent in any sense.
\end{example}

We make some simple observations. Firstly, the equality $\mathcal{F}(\overline{\Omega_{+}}, \overline{\Omega_{-}}) = \mathcal{F}^{*}(\overline{\Omega_{+}}, \overline{\Omega_{-}})$ is easily seen to be true, always. In particular, if $\Omega_{+}$ and $\Omega_{-}$ are closed sets, then the extremal problems are equivalent in any sense. Another observation is that if $f \in \mathcal{F}(\Omega_{+}, \Omega_{-})$, then taking into account the continuity of $f$ we have $f^{-1}(0, \infty) \subset \Int \Omega_{+} \subset \Omega_{+}$, and hence $\supp f_{+} \subset \overline{\Int{\Omega_{+}}}$. Similarly, $f^{-1}(- \infty, 0) \subset  \Int{\Omega_{-}} \subset \Omega_{-}$, and hence $\supp f_{-} \subset \overline{\Int{\Omega_{-}}}$. Thus, we have the following containments:

\begin{equation}\label{containments}
\begin{split}
\mathcal{F}^{*}(\Omega_{+}, \Omega_{-}) \subseteq \mathcal{F}(\Omega_{+}, \Omega_{-}) &= \mathcal{F}(\Int \Omega_{+}, \Int \Omega_{-})\\
& \subseteq \mathcal{F}(\overline{\Int \Omega_{+}}, \overline{\Int \Omega_{-}})\\
& = \mathcal{F}^{*}(\overline{\Int \Omega_{+}}, \overline{\Int \Omega_{-}})\\
& \subseteq \mathcal{F}^{*}(\overline{\Omega_{+}}, \overline{\Omega_{-}}) =  \mathcal{F}(\overline{\Omega_{+}}, \overline{\Omega_{-}}),
\end{split}
\end{equation}
and the corresponding relations between the extremal values:
\begin{equation}\label{inequalities}
\begin{split}
\mathcal{C}^{*}(\Omega_{+}, \Omega_{-}) \le \mathcal{C}(\Omega_{+}, \Omega_{-}) &=  \mathcal{C}(\Int \Omega_{+}, \Int \Omega_{-})\\
& \le \mathcal{C}(\overline{\Int{\Omega}_{+}}, \overline{\Int \Omega_{-} })\\
& = \mathcal{C}^{*}(\overline{\Int\Omega_{+}}, \overline{\Int \Omega_{-}})\\
&  \le \mathcal{C}^{*}(\overline{\Omega_{+}}, \overline{\Omega_{-}}) = \mathcal{C}(\overline{\Omega_{+}}, \overline{\Omega_{-}}).
\end{split}
\end{equation}

In some situations, the above containments and inequalities are equalities. This can easily be shown if $\Omega_{+}$ and $\Omega_{-}$ are closed. The next proposition shows that equality holds when $\Omega_{+}$ and $\Omega_{-}$ are boundary-coherent.

\begin{proposition}\label{equalconstants}
	Let $G$ be a LCA group and $\Omega_{+}, \Omega_{-}$ be boundary-coherent symmetric subsets of $G$ with $\Omega_{+}$ a neighbourhood of $0$. Then $\mathcal{F}^{*}(\overline{\Omega_{+}}, \overline{\Omega_{-}}) = \mathcal{F}(\Omega_{+}, \Omega_{-})$.
\end{proposition}

\begin{proof}
	
	Note that  $\mathcal{F}(\Omega_{+}, \Omega_{-}) \subset \mathcal{F}^{*}(\overline{\Omega_{+}}, \overline{\Omega_{-}})$ is always true. So, it is enough to show that \linebreak $\mathcal{F}^{*}(\overline{\Omega_{+}}, \overline{\Omega_{-}}) \subseteq \mathcal{F}(\Omega_{+}, \Omega_{-})$. Let $f \in \mathcal{F}^{*}(\overline{\Omega_{+}}, \overline{\Omega_{-}})$. The aim is to show that $f \in \mathcal{F}(\Omega_{+}, \Omega_{-})$. In other words, the aim is to show that $f^{-1}(0,\infty) \subset \Omega_{+}$ and $f^{-1}(- \infty, 0) \subset \Omega_{-}$. Now, on account of $\overline{\Omega_{+}} = \Omega_{+} \cup \partial{\Omega_{+}}$ and $\overline{\Omega_{-}} = \Omega_{-} \cup \partial{\Omega_{-}}$, it suffices to show that $f(x)\le 0$ for all $x \in \partial{\Omega_{+}}$ and that $f(x) \ge 0$ for all $x \in \partial{\Omega_{-}}$. To that end, let $x \in \partial \Omega_{+}$. Pick $\varepsilon > 0$ arbitrarily, and use the continuity of $f$ to obtain an open neighbhourhood $V$ of $x$ such that $f(y) \ge f(x) - \varepsilon$ for all $y \in V$. Since $x$ is in the boundary of $\ext \Omega_{+}$, there is an open set $U$ contained in $V$ that is fully contained in $\ext \Omega_{+}$. But $y \in \ext \Omega_{+}$ implies that $f(y) \le 0$. Hence $0 \ge f(y) \ge f(x) - \varepsilon$ for all $y \in U$. Since $\varepsilon$ is arbitrary, it follows that $f(x) \le 0$. Similarly, $f(x) \ge 0$ for all $x \in \partial \Omega_{-}$.

\end{proof}

\begin{remark} In the course of the proof of Theorem \ref{sigmacompactcase} we argued similarly to get that the limit function $f$ is in $ \mathcal{F}(\Omega_{+}, \Omega_{-})$. There it  is easy to see that $f\in \mathcal{F}^{*}(\overline{\Omega_{+}}, \overline{\Omega_{-}})$. So, the above Proposition can replace the argument there, showing $f \in \mathcal{F}(\Omega_{+}, \Omega_{-})$ under the condition of boundary-coherence.
\end{remark}

\begin{corollary} \label{equivalence}
Under the conditions of Proposition \ref{equalconstants}, the following equalities hold:
	\begin{equation*}
	\mathcal{F}(\Omega_{+}, \Omega_{-}) = \mathcal{F}(\Int \Omega_{+}, \Int \Omega_{-}) = \mathcal{F}(\overline{\Int \Omega_{+}}, \overline{\Int \Omega_{-}}) = \mathcal{F}^{*}(\overline{\Int \Omega_{+}}, \overline{\Int \Omega_{-}}) = \mathcal{F}^{*}(\overline{\Omega_{+}}, \overline{\Omega_{-}}).
	\end{equation*}
\end{corollary}

We also have equality of the extremal constants.

\begin{corollary} \label{equivalence2}
Under the conditions of Proposition \ref{equalconstants}, the following equalities hold:
	\begin{equation*}
	\mathcal{C}(\Omega_{+}, \Omega_{-}) = \mathcal{C}(\Int \Omega_{+}, \Int \Omega_{-}) = \mathcal{C}(\overline{\Int{\Omega}_{+}}, \overline{\Int \Omega_{-} }) = \mathcal{C}^{*}(\overline{\Int\Omega_{+}}, \overline{\Int \Omega_{-}}) = \mathcal{C}^{*}(\overline{\Omega_{+}}, \overline{\Omega_{-}}).
	\end{equation*}
\end{corollary}

\begin{example}
	
	Consider the Tur\'an problem for $\Omega = (-2, -1) \cup (-1, 1) \cup (1, 2)$. Observe that $\Omega$ is not boundary-coherent. By \cite[Theorem 7]{kolrev}, the extremal constant is $\mathcal{T}_{\mathbb{R}}(\Omega) = 1$, which is attained by the function $f = \mathbf{1} _{\left [-\frac{1}{2},\frac{1}{2}\right]} \ast \mathbf{1} _{\left [-\frac{1}{2},\frac{1}{2}\right]}$ appearing on the left of Figure \ref{extfunctions2}. On the other hand, $\overline{\Omega} = [-2,2]$, and the value of the extremal constant is   $\mathcal{T}_{\mathbb{R}}([-2,2]) = 2$, which is attained by the function $g = \frac{1}{2}\left(\mathbf{1}_{\left[-1,1\right]} \ast \mathbf{1}_{\left[-1,1\right]}\right)$  appearing on the right of Figure \ref{extfunctions2}. In particular, $\mathcal{T}_{\mathbb{R}}(\Omega) \ne \mathcal{T}_{\mathbb{R}}(\overline{\Omega})$.
	\begin{figure}[h]
		\begin{center}
		\begin{tikzpicture}
		\begin{axis}[scale = 0.5, ymajorticks=true, ytick = {1}, xtick ={-2,-1,1,2}, clip = false, xmin= -2.5, ymin = 0, xmax = 2.5, ymax = 1.25, axis lines=middle, xlabel = $x$, ylabel = $y$
		]
		\addplot[domain = 0:1]{(1-x)};
		\addplot[domain = -1:0]{(1+x)};
		\draw[very thick] (-2,0) -- (2,0);
		\addplot+[only marks,mark=*,mark options={scale=1, fill=white},text mark as node=true] coordinates {
			(-2,0) (-1,0) (1,0) (2,0)};
		\end{axis}
		\end{tikzpicture}
		\hskip 5pt
		\begin{tikzpicture}
		\begin{axis}[scale = 0.5, ymajorticks=true, ytick = {1},xtick ={-2,-1,1,2}, clip = false, xmin= -2.5, ymin = 0, xmax = 2.5, ymax = 1.25, axis lines=middle, xlabel = $x$, ylabel = $y$
		]
		\addplot[domain = 0:2]{(1-x/2)};
		\addplot[domain = -2:0]{(1+x/2)};
		\draw[very thick] (-2,0) -- (2,0);
		\end{axis}
		\end{tikzpicture}
		\end{center}
		\caption{Extremal functions for $\Omega$ and for $\overline{\Omega}$.}
		\label{extfunctions2}
	\end{figure}
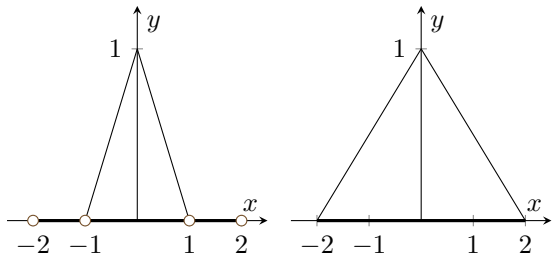

\end{example}

Notice that in Corollary \ref{equivalence} the set $\mathcal{F}^{*}(\Omega_{+}, \Omega_{-})$ is missing. Indeed, Example \ref{turanexample} demonstrates that this should be the case and that we can not do better under these conditions. On the other hand, this does not exclude, as demonstrated by Example \ref{turanexample}, the possibility that under the same conditions the term $\mathcal{C}^{*}(\Omega_{+}, \Omega_{-})$ should feature in the equalities in Corollary \ref{equivalence2}. We consider this problem separately. Our point of departure is the following proposition.

\begin{proposition}[See {\cite[ Proposition 3.1]{elena-szilard}}]\label{centrally-symmetric}
	Let $\Omega_{+}, \Omega_{-} \subset \mathbb{R}^d$ be $0$-symmetric convex sets with $\Omega_{+}$ a neighbourhood of $0$. Then $\mathcal{C}^{*}(\Int \Omega_{+}, \Int \Omega_{-}) = \mathcal{C}^{*}(\overline{\Omega_{+}}, \overline{\Omega_{-}})$.
\end{proposition}

This allows us to extend the equalities in Corollary \ref{equivalence2} as follows.

\begin{corollary}\label{star-like}
	If $\Omega_{+}, \Omega_{-} \subset \mathbb{R}^d$ are $0$-symmetric convex sets with $\Omega_{+}$ a neighbourhood of $0$, then the following equalities hold:
	\begin{equation*}
	\begin{split}
	\mathcal{C}^{\ast}(\Int \Omega_{+}, \Int \Omega_{-}) = \mathcal{C}^{*}(\Omega_{+}, \Omega_{-})  = \mathcal{C}(\Omega_{+}, \Omega_{-}) & = \mathcal{C}(\Int \Omega_{+}, \Int \Omega_{-}) = \mathcal{C}(\overline{\Int{\Omega}_{+}}, \overline{\Int \Omega_{-} }) \\ &= \mathcal{C}^{*}(\overline{\Int\Omega_{+}}, \overline{\Int \Omega_{-}})= \mathcal{C}^{*}(\overline{\Omega_{+}}, \overline{\Omega_{-}}).
	\end{split}
	\end{equation*}
	
\end{corollary}

\begin{proof}
	Since convex sets are boundary-coherent, Corollary \ref{equivalence2} applies to give
	\begin{equation}\label{chain}
	\mathcal{C}(\Omega_{+}, \Omega_{-}) = \mathcal{C}(\Int \Omega_{+}, \Int \Omega_{-}) = \mathcal{C}(\overline{\Int{\Omega}_{+}}, \overline{\Int \Omega_{-} }) = \mathcal{C}^{*}(\overline{\Int\Omega_{+}}, \overline{\Int \Omega_{-}}) = \mathcal{C}^{*}(\overline{\Omega_{+}}, \overline{\Omega_{-}}).
	\end{equation}
	By Proposition \ref{centrally-symmetric}, we have
	\begin{equation}\label{convex}
	\mathcal{C}^{*}(\overline{\Omega_{+}}, \overline{\Omega_{-}}) = \mathcal{C}^{*}(\Int \Omega_{+}, \Int \Omega_{-}) \le \mathcal{C}^{\ast}(\Omega_{+}, \Omega_{-}) \le \mathcal{C}(\Omega_{+}, \Omega_{-}) = \mathcal{C}^{\ast}(\overline{\Omega_{+}}, \overline{\Omega_{-}}),
	\end{equation}
	and the corollary is proved.
\end{proof}

At this point, a natural goal is to generalise Proposition \ref{centrally-symmetric}, and, as a result, generalise Corollary~\ref{star-like}. Firstly, we introduce the following function class:
\begin{equation*}
\mathcal{F}^{*}_{c}(\Omega_{+}, \Omega_{-}) := \{f \in \mathcal{F}_{G}(\Omega_{+}, \Omega_{-}): \supp f \Subset G \}.
\end{equation*}
Define the corresponding extremal value:
\begin{equation*}
\mathcal{C}_{c}(\Omega_{+}, \Omega_{-}):= \sup _{f \in \mathcal{F}^{*}_{c}(\Omega_{+}, \Omega_{-})}\int_{G}f\mbox{d}\lambda_{G}.
\end{equation*}

Following \cite[Theorem 2.1]{elena-szilard}, we have the following lemma.

\begin{lemma}[See {\cite[Theorem 2.1]{elena-szilard}}] \label{equileamma}
	If $\Omega_{+}$ and $\Omega_{-}$ are $0$-symmetric subsets of a LCA group $G$ with $\Omega_{+}$ a neighbourhood of $0$, then $\mathcal{C}_{c}^{*}(\Omega_{+}, \Omega_{-}) = \mathcal{C}^{*}(\Omega_{+}, \Omega_{-})$.
\end{lemma}

\begin{theorem} \label{starbodytheorem}
	Let $G$ be a LCA group which is also a topological vector space (TVS) over $\mathbb{R}$. Let $\Omega_{+}$ and $\Omega_{-}$ be bounded $0$-symmetric sets in $G$ such that $r\overline{\Omega_{+}} \subset \Int \Omega_{+} $ and $r \overline{\Omega_{-}} \subset \Int \Omega_{-}$ for all $r \in [0,1)$. Then $\mathcal{C}^{*}(\Int \Omega_{+}, \Int \Omega_{-}) = \mathcal{C}^{*}(\overline{\Omega_{+}}, \overline{\Omega_{-}})$.
\end{theorem}

\begin{proof}
	Clearly, $\mathcal{C}^{*}(\Int \Omega_{+}, \Int \Omega_{-}) \le \mathcal{C}^{*}(\overline{\Omega_{+}}, \overline{\Omega_{-}})$, so we only prove the reverse inequality. Let us take $\varepsilon >0$ and $f \in \mathcal{F}^{*}_{c}(\overline{\Omega_{+}}, \overline{\Omega_{-}})$ with $\int_{G}f\mbox{d}\lambda_{G} \ge \mathcal{C}^{*}_{c} (\overline{\Omega_{+}}, \overline{\Omega_{-}}) - \varepsilon = \mathcal{C}^{*}(\overline{\Omega_{+}}, \overline{\Omega_{-}}) - \varepsilon$. The last equality is by Lemma \ref{equileamma}. Take any $R > 1$ and put $r:= 1/R < 1$. Consider the function $h_{R}(g) := f(Rg)$. Obviously, $\supp (h_{R})_+ \subset r \overline{\Omega_{+}} \subset \Int \Omega_{+}$ and $\supp (h_{R})_{-} \subset r \overline{\Omega_{-}} \subset \Int \Omega_{-}$. Furthermore, $h_{R}$ is a compactly supported positive definite continuous function satisfying $h_{R}(0) = 1$. Hence $h_{R} \in \mathcal{F}^{\ast}_{c}(\Int \Omega_{+}, \Int \Omega_{-})$. Now, $h_{R}(g) = f(Rg) \to f(g)$ as $R \to 1+$. Let $0 < \delta < 1$ and consider the set
	\begin{equation*}
	K:= \{rg: 1 - \delta \le r \le 1, g \in \supp f \} = \left[1 - \delta, 1 \right] \cdot \supp f.
	\end{equation*}
	Observe that $K$, being a continuous image of the compact set $[1- \delta,1] \times \supp f \subset \mathbb{R} \times G$ under the multiplication map $ \mathbb{R} \times G \ni (r, g) \mapsto rg \in G$ is a compact subset of $G$, and hence has finite Haar measure. For $1 < R < \frac{1}{1 - \delta}$, we have that $\supp h_{R} \subset K$, so that $|h_{R}| \le \mathbf{1}_{K} \in L^1(G)$. Hence, by Lebesgue's Dominated Convergence Theorem, we have
	\begin{equation*}
	\lim _{R \to 1+}\int_{G}h_{R}\mbox{d}\lambda_{G} = \int_{G}f\mbox{d}\lambda_{G}.
	\end{equation*}
	In other words, there exists $R_{0} > 1$ such that
	\begin{equation*}
	\left|\int_{G}h_{R}\mbox{d}\lambda_{G} - \int_{G}f\mbox{d}\lambda_{G} \right| < \varepsilon, \mbox{ for all } R \in (1, R_{0}).
	\end{equation*}
	Then, for all $R \in (1, R_{0})$, keeping in mind that $h_{R} \in \mathcal{F}^{\ast}(\Int \Omega_{+}, \Int \Omega_{-})$, we have that
	\begin{equation*}
	\mathcal{C}^{*}(\Int \Omega_{+}, \Int \Omega_{-}) \ge \int_{G}h_{R}\mbox{d}\lambda_{G} > \int_{G}f\mbox{d}\lambda_{G} - \varepsilon \ge \mathcal{C}^{*}(\overline{\Omega_{+}}, \overline{\Omega_{-}}) - 2 \varepsilon.
	\end{equation*}
	Since $\varepsilon > 0$ was arbitrary, we have $\mathcal{C}^{*}(\Int \Omega_{+}, \Int \Omega_{-}) \ge \mathcal{C}^{*}(\overline{\Omega_{+}}, \overline{\Omega_{-}})$ as desired.
\end{proof}

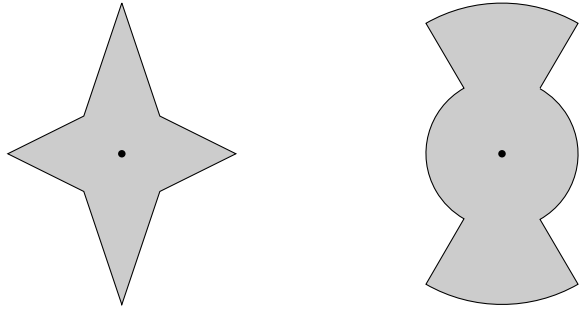
\begin{figure}[htbp]
	\centering
	
	\begin{tikzpicture}\label{starsets}
	
	\draw[thin, fill=gray!40] (0,2) -- (0.5,0.5) -- (1.5,0) -- (0.5,-0.5) -- (0,-2) -- (-0.5,-0.5) -- (-1.5,0) -- (-0.5,0.5) -- (0,2);
	\filldraw[thick, fill=black] (0,0) circle [radius = 1pt];
	
	\filldraw[thin,fill=gray!40] (5.5, 0.87) -- (6,1.73) arc (60:120:2) -- (4.5,0.87) arc(120:240:1) -- (4,-1.73) arc (240:300:2) -- (5.5,-0.87) arc(300:420:1);
	\filldraw[thick, fill=black] (5,0) circle [radius = 1pt];

	\end{tikzpicture}
	\caption{Two open, $0$-symmetric, star-shaped sets. The left set satisfies the condition $r \overline{X} \subset {\rm int}X$, $0 \le r < 1$; the right set does not satisfy this condition --- it is even true that $r \overline{X} \not\subset X$.}
\end{figure}

After proving Theorem \ref{starbodytheorem}, we noticed the paper \cite{mavroudis} where the following theorem was proved.

\begin{theorem}[See {\cite[Theorem 1]{mavroudis}}]\label{mavroudis}
	
	Let  $\Omega \subset \mathbb{R}^{d}$ be an open, bounded, strictly star-shaped\,\footnote{The authors call an open set $\Omega$ strictly star-shaped if $r \overline{\Omega} \subset \Omega$ for all $r \in [0,1)$.} and symmetric set, and assume that $\varepsilon > 0$ and f is a continuous, positive definite function which vanishes outside $\Omega$ (i.e. $\supp f \subset \overline{\Omega}$). Then there is a continuous positive definite function $F$, with $\supp F \subset \Omega$ , such that $\|f - F \|_{\infty} < \varepsilon$.
	
\end{theorem}

Theorem \ref{mavroudis} and its proof are essentially the same as the argument used in Theorem \ref{starbodytheorem}, so we give credit to \cite{mavroudis} while observing that our work naturally leads us to consider a result along these lines.

Now we have the following general form of Corollary \ref{star-like}.

\begin{corollary}\label{bodiescorollary}
	Let $G$ be a LCA group which is also a TVS over $\mathbb{R}$. Let $\Omega_{+}$ and $\Omega_{-}$ be bounded $0$-symmetric sets of $G$ such that $r\overline{\Omega_{+}} \subset \Int \Omega_{+} $ and $r \overline{\Omega_{-}} \subset \Int \Omega_{-}$ for all $r \in [0,1)$. Then the following equalities hold:
	\begin{equation*}
	\begin{split}
	\mathcal{C}^{\ast}(\Int \Omega_{+}, \Int \Omega_{-}) = \mathcal{C}^{*}(\Omega_{+}, \Omega_{-}) = \mathcal{C}(\Omega_{+}, \Omega_{-}) &= \mathcal{C}(\Int \Omega_{+}, \Int \Omega_{-}) = \mathcal{C}(\overline{\Int{\Omega}_{+}}, \overline{\Int \Omega_{-} })\\
	&= \mathcal{C}^{*}(\overline{\Int\Omega_{+}}, \overline{\Int \Omega_{-}})
	 = \mathcal{C}^{*}(\overline{\Omega_{+}}, \overline{\Omega_{-}}).
	\end{split}
	\end{equation*}
\end{corollary}

\section*{Acknowledgements}

E.~E.~Berdysheva would like to thank the Isaac Newton Institute for Mathematical Sciences, Cambridge, for support and hospitality during the programme ``Discretization and recovery in high-dimensional spaces'' where part of the work on this paper was undertaken. This work was supported by EPSRC grant no EP/R014604/1. M.~D.~Ramabulana is grateful for the financial support from the Shuttleworth Postgraduate Scholarship of the University of Cape Town. Sz.~Gy.~R\'{e}v\'{e}sz was supported in part by the Hungarian National Research,
Development and Innovation Office, project \#s K-147153 and K-146387.


\begin{thebibliography}{99}
	
	\bibitem{elena-arestov-2} V. V. Arestov and E. E. Berdysheva, \textit{Tur\'{a}n's problem for a class of polytopes}, East J. Approx. \textbf{8} (2002) 3, 381--388.
	
	\bibitem{elena-szilard} E. E. Berdysheva, Sz. Gy. R\'{e}v\'{e}sz, \textit{Delsarte’s Extremal Problem and Packing on Locally Compact Abelian Groups}, Annali della Scuola Normale di Pisa – Classe di Scienze (5) \textbf{XXIV} (2023), 1007--1052.
	
	\bibitem{brezis} H. Brezis, \say{Functional Analysis, Sobolev Spaces, and Partial Differential Equations}, Springer-Verlag, New York (2010).
	
	\bibitem{caratheodory} C. Carath\'{e}odory, \textit{\"{U}ber den Variabilit\"{a}tsbereich der Fourier’schen Konstanten von positiven harmonischen Funktionen}, Rendiconti del Circolo Matematico di Palermo, \textbf{32} (1911), 193--217.
	
	\bibitem{cohnlaatsal} H. Cohn, D. De Laat, and A. Salmon, \textit{Three-point bounds for sphere packing}, arXiv:2206.15373.
	
	\bibitem{deitmar} A. Deitmar, S. Echterhoff, \say{Principles of Harmonic Analysis}, Second Edition, Springer (2014).
	
	\bibitem{domar} Y. Domar, \textit{An extremal problem for positive definite  functions}, J. Math. Anal. Appl. \textbf{52} (1975), 56--63.
	
	\bibitem{folland} G. Folland, \say{A Course in Abstract Harmonic Analysis}, Second Edition, Taylor \& Francis Group (2016).
	
	\bibitem{fejer} L. Fej\'{e}r, \textit{\"{U}ber trigonometrische Polynome}, J. Reine Angew. Math. \textbf{146} (1916), 53-82.
	
	\bibitem{marcell-zsuzsa} M. Ga\'{a}l and Zs. Nagy-Csiha, \textit{On the Existence of an Extremal Function
		in the Delsarte Extremal Problem}, Mediterr. J. Math, \textbf{17}:190 (2020).
	
	\bibitem{gorbachev1} D. V. Gorbachev, \textit{An extremal problem for periodic functions with supports in the ball}, Math. Notes (3) \textbf{69} (2001), 313--319.
	
	\bibitem{kolrev0} M. N. Kolountzakis and Sz. Gy. Révész, \textit{ On a problem of Tur\'{a}n about positive definite functions}, Proc. Amer. Math. Soc. \textbf{131} (2003), 3423--3430.
	
	\bibitem{kolrev} M. N. Kolountzakis and Sz. Gy. Révész, \textit{Tur\'{a}n’s extremal problem for positive definite functions on groups}, J. London Math. Soc. \textbf{74} (2006), 475--496.
	
	
	\bibitem{mavroudis} P. G. Mavroudis, \textit{On the approximation of positive definite functions by like functions of smaller support}, Bull. Math. Sci. \textbf{3} (2013), 287--298.
	
	\bibitem{ramabulana} M. D. Ramabulana, \textit{On the existence of an extremal function for the Delsarte extremal problem}, Analysis Mathematica, to appear. See also as arxiv: 2407.04410.
	
	
	\bibitem{turan-szilard} Sz. Gy. R\'{e}v\'{e}sz, \textit{Tur\'{a}n's extremal problem on locally compact Abelian groups}, Anal. Math. \textbf{37} (2011), 15--50.
	
	\bibitem{sasvari} Z. Sasv\'{a}ri, \say{Positive definite and definitizable functions}, Vol. 2, Akademie Verlag, Berlin (1994).
	
	\bibitem{siegel} C. L. Siegel, \textit{ \"{U}ber Gitterpunkte in konvexen K\"{o}rpern und ein damit zusammenh\"{a}n\"{g}endes Extremalproblem}, Acta Math. \textbf{65} (1935), 307--323.
	
	
	
\end{thebibliography}
\end{document}